\documentclass[12pt]{amsart}
\usepackage{amssymb}
\usepackage{amsmath,amscd}
\usepackage{amsfonts,latexsym,verbatim,amscd,mathrsfs,color,array}
\usepackage[all,cmtip]{xy}
\usepackage{mathrsfs}
\usepackage{multirow}
\usepackage{graphicx}
\usepackage{amsfonts, amsthm}
\usepackage{bbm}
\usepackage[hmargin=1in,vmargin=1in]{geometry}
\usepackage{mathdots}
\usepackage{stmaryrd}
\usepackage{MnSymbol}
\usepackage{bbm}
\usepackage[hidelinks]{hyperref}

\allowdisplaybreaks

\def\XXint#1#2#3{{\setbox0=\hbox{$#1{#2#3}{\int}$ }
		\vcenter{\hbox{$#2#3$ }}\kern-.6\wd0}}

\DeclareMathOperator*{\Osc}{\osc}

\usepackage{accents}
\newlength{\dhatheight}

\newcommand{\osc}{\operatorname{osc}}
\DeclareMathSymbol{\intprod}{\mathbin}{MnSyC}{'270}
\newcommand{\LB}{\left[}
\newcommand{\RB}{\right]}
\newcommand{\LA}{\left\langle}
\newcommand{\RA}{\right\rangle}

\newcommand{\N}{{\mathbb N}}
\newcommand{\C}{{\mathbb C}}

\newcommand{\R}{{\mathbb R}}

\newcommand{\eps}{{\varepsilon}}

\newcommand{\T}{{\mathcal T}}

\newcommand{\p}{{\partial}}

\newcommand{\That}{{ \widehat{\T} }}

\newtheorem{thm}{Theorem}[section]
\newtheorem{lemma}[thm]{Lemma}
\newtheorem*{lemma*}{Lemma}
\newtheorem{prop}[thm]{Proposition}

\newtheorem{cor}[thm]{Corollary}

\newtheorem*{conj*}{Conjecture}

   \newtheoremstyle{others}
     {3pt}
     {2pt}
     {}
     {}
     {\bf}
     {.}
     {.5em}
     {}

\theoremstyle{others}
\newtheorem{rmk}[thm]{Remark}
\newtheorem*{rmk*}{Remark}
\newtheorem{defn}[thm]{Definition}

\numberwithin{equation}{section}

\begin{document}

\title{Weighted \L ojasiewicz inequalities
and regularity of harmonic map flow}
\author{Alex Waldron}
\address{University of Wisconsin, Madison}
\email{waldron@math.wisc.edu}

\begin{abstract}
At a finite-time singularity of harmonic map flow in the critical dimension, we show that a \L ojasiewicz inequality between the quantities appearing in Struwe's monotonicity formula implies continuity of the body map and the no-neck property for bubble-tree decompositions. We prove such an inequality when the target is $S^2,$ yielding both properties in this case.
\end{abstract}

\maketitle

\thispagestyle{empty}



\vspace{-5mm}

\section{Introduction}

\L ojasiewicz(-Simon) inequalities play a key role in the theory of geometric evolution equations; they rule out the notorious non-uniqueness phenomenon for subsequential limits at infinite time \cite{leonsimon}.
In recent years, most prominently in the work of Colding and Minicozzi, \L ojasiewicz inequalities suitable for finite-time singularity analysis have become available.
These have led to uniqueness theorems for the most basic multiplicity-one blowup limits of mean curvature flow, namely, round cylinders \cite{coldingminicozziuniquenesslojasiewicz} and asymptotically conical shrinkers \cite{chodoshschulzeuniquenessofasymptoticallyconical}, as well as mean-convex cylindrical shrinkers \cite{zhulojasiewiczcylindrical, zhulojasiewiczmeanconvex}.

\L ojasiewicz inequalities are also strongly tied to 
questions of regularity. Simon used his classical inequality in the elliptic setting to obtain a strong Hausdorff-dimension estimate and rectifiability of the singular set for area-minimizing currents \cite{simonrectifiability}. Colding and Minicozzi have applied their results to obtain improved regularity of singular sets for mean curvature flow with generic (i.e. cylindrical) singularities at finite time \cite{coldingminicozzisingularset}. They have also demonstrated that differentiability of the ``arrival time'' in mean-convex mean curvature flow depends on a \L ojasiewicz inequality \cite{coldingminicozzidifferentiabilityofthearrivaltime}.

This paper establishes a link between a different regularity question and a \L ojasiewicz inequality. 
The question concerns the so-called ``body map'' $u(T) = \lim_{t \nearrow T} u(t)$ of a solution of harmonic map flow $u : \Sigma \times \LB 0, T \right) \to N$ in the critical dimension, $\dim\left( \Sigma \right) = 2.$ This map is by definition of class $H^1,$ but may fail to be continuous in certain cases, as shown by Topping \cite{toppingwindingbehavior}. Meanwhile, Topping conjectured that if $N$ is real-analytic then the body map must be continuous \cite[p. 286]{toppingwindingbehavior}.

Our main result, Theorem \ref{thm:mainthm} below, is that continuity of the body map follows from the validity of a \L ojasiewicz inequality 
between the Gaussian-weighted quantities appearing in Struwe's monotonicity formula \cite{struwehmfhigherd}. We prove such an inequality, Theorem \ref{thm:firstloj}, in the case that the target manifold is the round 2-sphere, allowing us to 
settle the fundamental case of the above question.\footnote{The continuity question for $N = S^2$ was raised again recently by Jendrej, Lawrie, and Schlag \cite{jendrejlawrieschlag}.} 
Our main estimate also gives the ``no-neck'' property of the bubble-tree decomposition along any sequence approaching the singular time.

The weighted \L ojasiewicz inequality for maps to the 2-sphere is obtained by localizing one of the inequalities proved recently by the author \cite{lojasiewiczonS2}.
The localization is enabled by a weighted Poincar\' e inequality, Lemma \ref{lemma:Poincare}, similar to those used in Type-I singularity analysis; this inequality turns out to be particularly strong in the critical dimension. We use the inequality in two places: once in the localization, Theorem \ref{thm:firstloj}, and later in the proof of main theorem, cf. (\ref{psiintegralbound}).

An interesting feature of the proof is that we do not obtain \emph{H\"older} continuity of the body map, as in the case of a ``strict'' type-II blowup \cite{stricttypeII}, unless the \L ojasiewicz inequality holds with the optimal exponent ($\alpha = 2$). In the case of a fractional exponent ($1 < \alpha < 2$), we obtain the oscillation bound
\begin{equation*}
\Osc_{B_r} u(T) = O \left| \log r \right|^{\frac{\alpha - 1}{\alpha - 2}}.
\end{equation*}
See Theorem \ref{thm:mainthm}-Corollary \ref{cor:oscillation} for detailed statements. 

\vspace{5mm}

\section{Weighted \L ojasiewicz inequality}

\subsection{Weighted Poincar{\'e} inequality} Let $(\Sigma, g)$ be a Riemannian surface and $(N,\LA \cdot, \cdot \RA)$ any compact Riemannian manifold. Given a map $u : \Sigma \to N,$ denote its stress-energy tensor by
\begin{equation}\label{stressdef}
S_{ij} = \LA \p_i u, \p_j u \RA - \frac12 g_{ij} |du|_g^2.
\end{equation}
We have
\begin{equation}\label{stressdiv}
\nabla^i S_{ij} = \LA \T (u) , \p_j u \RA,
\end{equation}
see e.g. \cite[(2.5)]{songwaldron}.
Let $X^i$ be any conformal Killing field, satisfying
$$\nabla^{i} X^j + \nabla^j X^i = \mu g^{ij}$$
for some function $\mu.$
Note that $g^{ij} S_{ij} = 0$ in dimension two. Contracting the identity (\ref{stressdiv}) with $X,$ we have
\begin{equation}\label{stressenergyidentity}
\nabla^i \left( X^j S_{ij} \right) = \LA \T (u), X^j \p_j u \RA.
\end{equation}
We now specialize to the case $\Sigma = \R^2$ with the flat metric; the general case requires only small modifications. 
Let $x = x^i \p_i$ be the radial vector field and $r = \sqrt{(x^1)^2 + (x^2)^2}$ the radial coordinate in $\R^2.$ Define the weighted $L^2$-norm
\begin{equation}\label{bartaudef}
\| \alpha \|_\tau = \sqrt{ \int_{\R^2} |\alpha |^2 \, e^{-\frac{r^2}{4\tau} } \, dV }.
\end{equation}
Let
\begin{equation*}\label{Thatpredefn}
\That_\tau (u) := \T(u) - \frac{x}{2 \tau  } \intprod du.
\end{equation*}

\begin{lemma}\label{lemma:Poincare} For any $W^{2,2}$ map $u: \R^2 \to N,$ we have
\begin{equation}\label{twistedtensiondominatesrdu}
\left\| r du \right\|_\tau \leq 4 \tau \left\|\That_\tau(u) \right\|_\tau
\end{equation}
and
\begin{equation}\label{Tdominance}
\| \T(u) \|_{\tau} \leq 3 \left\| \That_\tau(u) \right\|_{\tau}.
\end{equation}
\end{lemma}
\begin{proof}
We take $X = x$ in (\ref{stressenergyidentity}) and integrate by parts against $e^{-\frac{r^2}{4 \tau} },$ 
to obtain
\begin{equation*}
\frac{1}{2\tau} \int x^i x^j S_{ij} e^{-\frac{r^2}{4 \tau} } \, dV = \int \LA \T(u), x^j \p_j u \RA e^{-\frac{r^2}{4 \tau} } \, dV.
\end{equation*}
Unpacking the definition of the stress-energy tensor (\ref{stressdef}), we have
\begin{equation*}
\begin{split}
-  \frac{1}{4 \tau} \int r^2 |du|^2 e^{-\frac{r^2}{4 \tau} } \, dV & = \int \LA \T(u), x^j \p_j u \RA e^{-\frac{r^2}{4 \tau} } \, dV - \int \LA \frac{ x^i  }{ 2\tau } \p_i u, x^j \p_j u \RA e^{-\frac{r^2}{4 \tau} } \, dV \\
& = \int \LA \That_\tau(u), x^j \p_j u \RA e^{-\frac{r^2}{4 \tau} } \, dV.
\end{split}
\end{equation*}
Applying Cauchy-Schwarz, we obtain
$$\frac{1}{4 \tau} \left\| r du \right\|_\tau^2 \leq \left\| \That_\tau(u) \right\|_{\tau} \| x \intprod du \|_\tau.$$
Since $|x \intprod du| \leq r |du|,$ this implies (\ref{twistedtensiondominatesrdu}).
Then (\ref{Tdominance}) follows from (\ref{twistedtensiondominatesrdu}) and the triangle inequality.
\end{proof}

\vspace{5mm}

\subsection{Weighted \L ojasiewicz inequality for maps to $S^2$}

Let
$$\Phi(u) := \frac12 \int |du|^2 e^{ \frac{-r^2}{4} } \, dV$$
and
$$\That(u) : = \That_1(u) = \T(u) - \frac{x}{2 } \intprod du.$$

\begin{thm}\label{thm:firstloj}
Fix $k \in \N$ and $\beta > 0.$ Given any map $u : \R^2 \to S^2$ with $E(u) \leq 4 \pi k,$ there exists $n \in \{0, \ldots, k\}$ such that
$$|\Phi(u) - 4\pi n | \leq C_{k, \beta} \| \That(u) \|_1^{2 - \beta},$$
where $\| \cdot \|_{1}$ is the weighted $L^2$-norm defined by (\ref{bartaudef}).
\end{thm}
\begin{proof} 

By the weighted Poincar\'e inequality of Lemma \ref{lemma:Poincare}, we have
$$ \int r^2 |du|^2 e^{-\frac{r^2}{4}}  + \| \T(u) \|^2_{L^2(B_2)} \leq C \| \That(u) \|_1^2.$$
In particular, this gives $ E(u, B_2 \setminus B_1) \leq C\| \That(u) \|_1^2.$ We can assume without loss of generality that $ \| \That(u) \|^2_1 \leq \eps_0 / C,$ where $\eps_0$ is the standard $\eps$-regularity constant for maps to $S^2.$ 

Let $0 \leq \lambda \leq 1$ be the smallest radius such that
$$\int_{B_2 \setminus B_\lambda} |du|^2 \leq \eps_0.$$
We have
$$\lambda^2 \int_{B_2 \setminus B_\lambda} |du|^2 \leq \int r^2 |du|^2 e^{-\frac{r^2}{4}} \, dV \leq  C \| \That(u) \|_1^2.$$
This gives
$$\lambda \leq  C \| \That(u) \|_1.$$
Now, \cite[Corollary 1.3]{lojasiewiczonS2} reads
\begin{equation}\label{quantitativeloj:cestimate}
\begin{split}
&  \left| E(u, B_1) - 4 \pi n \right| \\
 & \qquad  \leq C_{k,\beta} \left( \| \T(u) \|_{L^2(B_2)}^2 + E(u, B_{2} \setminus B_1) + \lambda^{1 - \beta} \left( \| \T(u) \|_{L^2(B_2)} + \sqrt{E(u, B_{2} \setminus B_1)} \right) \right).
 \end{split}
\end{equation}
Inserting the above estimates into the RHS of (\ref{quantitativeloj:cestimate}), we get
$$\left| E(u, B_1) - 4 \pi n \right| \leq C \| \That(u) \|_1^{2 - \beta}.$$
We further have
$$\left| E(u, B_1) - \frac12 \int_{B_1} |du|^2 e^{\frac{-r^2}{4}} \right| \leq \int_{B_1} r^2 |du|^2 \leq C \| \That(u) \|_1^{2}.$$
Finally, we have
$$\int_{\R^2 \setminus B_1} |du|^2 e^{\frac{-r^2}{4}} \leq \int_{\R^2 \setminus B_1} r^2 |du|^2 e^{\frac{-r^2}{4}} \leq C \| \That(u) \|_1^{2}.$$
Combining these, we obtain the desired estimate.
\end{proof}

\vspace{5mm}

\section{Monotonicity formulae in dimension two}

Recall the \emph{pointwise energy identity} for harmonic map flow,
\begin{equation}\label{pointwiseenergy}
\frac12 \frac{\p }{\p t} |du|^2 + |\T(u)|^2 = \nabla^i \nabla^j S_{ij},
\end{equation}
see \cite[(2.7)]{songwaldron}.

\begin{lemma} 
Let
$$w = w(r,t) = \exp \left( -{\frac{\phi(x)}{2\tau(t)}} \right).$$
For a solution $u(t) : \Sigma \to N$ of harmonic map flow, we have
\begin{equation*}
\begin{split}
\frac12 \frac{\p}{\p t} \left( |du|^2 w \right) + \left| \T(u) - \frac{\nabla \phi}{2 \tau} \intprod du \right|^2 w & = \nabla^i \nabla^j \left( S_{ij} w \right) \\
& \qquad + \frac{w}{2\tau} \nabla^i \nabla^j \phi S_{ij} + \frac{\left( |\nabla \phi|^2 + 2 \phi \tau'(t) \right) w}{8\tau^2}  |du|^2.
\end{split}
\end{equation*}
\end{lemma}
\begin{proof}
Multiplying the pointwise energy identity (\ref{pointwiseenergy}) by $w,$ 
we obtain
\begin{equation*}
\begin{split}
\frac12 \frac{\p}{\p t} \left( |du|^2 w \right) - \frac{\tau'(t) }{4\tau^2} \phi |du|^2 w + |\T(u)|^2 w & = \nabla^i \nabla^j \left( S_{ij} w \right) - 2 \nabla^i w \nabla^j S_{ij} - \nabla^i \nabla^j w S_{ij} \\
& = \nabla^i \nabla^j \left( S_{ij} w \right) + 2\LA \T(u) , \frac{\nabla \phi \intprod du}{2\tau}  \RA w \\
&  \quad - \frac{ \left(2\tau \nabla^i \nabla^j \phi + \nabla^i \phi \nabla^j \phi \right) }{4\tau^2} S_{ij} w.
\end{split}
\end{equation*}
Recalling the definition of the stress-energy tensor (\ref{stressdef}), we can rearrange to obtain the desired formula.
\end{proof}

We have the following versions of Struwe's monotonicity formula in the critical dimension.

\begin{thm}\label{thm:monotonicities} Let $\Sigma = \R^2$ and take
$$\Phi_{\tau}(u) = \frac12 \int |du|^2 e^{-\frac{r^2}{4 \tau}} \, dV, \quad \Psi_{\tau}(u) = \frac12 \int r^2 |du|^2 e^{-\frac{r^2}{4 \tau}} \, dV.$$
We have
\begin{equation}\label{Phimonotonicity}
\frac{d}{dt} \Phi_{\tau}(u(t)) + \| \That_{\tau}(u(t)) \|_{\tau}^2 = \frac{1 + \tau'(t)}{4 \tau^2} \Psi_\tau(u(t))
\end{equation}
and
\begin{equation}\label{Psimonotonicity}
\frac{d}{dt} \Psi_{\tau}(u(t)) + \| r \That_{\tau}(u(t)) \|_{\tau}^2 = \frac{1 + \tau'(t)}{8 \tau^2} \int r^4 |du|^2 e^{-\frac{r^2}{4 \tau}} \, dV.
\end{equation}
\end{thm}
\begin{proof} Take $\phi = \frac{r^2}{2}$ in the previous Lemma, so that $w = e^{\frac{-r^2}{4\tau}}.$
We have
\begin{equation}\label{pointwisemonotR2}
\begin{split}
\frac12 \frac{\p}{\p t} \left( |du|^2 w \right) + \left| \That_{\tau}(u) \right|^2 w & = \nabla^i \nabla^j \left( S_{ij} w \right) + \frac{\left( 1 + \tau'(t) \right)r^2}{8 \tau^2} |du|^2 w.
\end{split}
\end{equation}
The identity (\ref{Phimonotonicity}) follows by integrating (\ref{pointwisemonotR2}) in space. The identity (\ref{Psimonotonicity}) follows by integrating by parts against $r^2$ and using $\nabla^i\nabla^j r^2 = 2 \delta^{ij}$ together with tracelessness of $S_{ij}.$
\end{proof}

\vspace{5mm}

\section{Estimates on energy density and oscillation}

\begin{lemma}\label{lemma:rduestimate}
Fix $0 < R \leq \sqrt{T},$ $0 < \lambda \leq \frac12,$ and $\kappa > 0.$ 
Let $u : \R^2 \times \LB 0, T \right) \to N$ be a classical solution of harmonic map flow.
Suppose that $R^2 \leq T_1 \leq T$ is such that 
\begin{equation}\label{rduestimate:smallnessassumption}
\sup_{T_1 - R^2 \leq t \leq T_1 - (R\lambda)^2} \frac{\Psi_{T_1 - t} \left( u(t) \right)}{T_1 - t} \leq \eps_0,
\end{equation}
where $\eps_0 > 0$ depends on the geometry of $N.$ Letting
\begin{equation}\label{etadef}
\eta : = \frac{\sqrt{ \Psi_{4R^2} \left( u(T_1 - R^2) \right) } }{R}  
\end{equation}
and
 \begin{equation}\label{psidefn}
\psi \left( \log \frac{R}{\sqrt{T_1 - t} } \right) : = \sqrt{\frac{\Psi_{T_1 - t} \left( u(t) \right)}{T_1 - t} },
\end{equation}
we have
\begin{equation*}
r |du(x,t)| \leq C_{\kappa} \left( \int_0^{\log \frac{R}{r}} \psi\left( \sigma \right) e^{\kappa \left( \sigma - \log \frac{R}{r} \right)} \, d\sigma  + \eta \left( \frac{r}{R} \right)^\kappa \right)
\end{equation*}
for 
$\lambda R \leq r \leq \min \left\{ \left( C_{N, \kappa} \sqrt[\kappa]{\eta} \right)^{-1} , 1 \right\} R$ and
$T_1 - \frac14 r^2 \leq t < T_1.$ 
\end{lemma}
\begin{proof}
We may assume without loss of generality that $R = T_1 = 1.$ Let $s = - \frac12 \log \left( 1 - t \right),$ and define
$$\bar{\psi}\left(s \right) := \sqrt{ \frac{\Psi_{4(1 - t)}(u(t))}{1 - t} }.$$
We claim that
\begin{equation}\label{psiclaim}
\bar{\psi} \left( s \right) \leq \eta e^{-\kappa s} + \frac{e^{4 + 4\kappa} }{\kappa} \int_0^{s } \psi \left( \sigma \right)e^{\kappa \left(\sigma - s \right)} \, d\sigma
\end{equation}
for all $0 \leq s \leq -\log \lambda.$ 

Applying the previous theorem with $\tau(t) = 4(1 - t),$ we obtain
\begin{equation}\label{psi'(t)}
\begin{split}
\frac{d \bar{\psi}^2 }{dt} 
& \leq \frac{1}{1 - t} \left( \bar{\psi}^2(t) - \frac{3}{8 \tau^2} \int r^4 |du|^2 e^{-\frac{r^2}{4\tau}} \, dV \right)
\end{split}
\end{equation}
Note that since $e^{a-3x} > a - 3x,$ we have 
$$3x e^{-x} > a e^{-x} - e^{a - 4x}.$$
Taking $a = 4 + 4\kappa$ and $x = \frac{r^2}{4 \tau},$ and multiplying by $-\frac{r^2}{2\tau},$ we obtain
$$-\frac{3r^4}{8 \tau^2} e^{-\frac{r^2}{4 \tau} } < \frac{r^2}{2 \tau} \left( e^{4 + 4 \kappa -\frac{r^2}{\tau} } - (4 + 4 \kappa) e^{-\frac{r^2}{4 \tau} } \right).$$
Using this in the integrand of the second term on the RHS of (\ref{psi'(t)}), we get
\begin{equation*}
\begin{split}
(1 -t) \frac{d\bar{\psi}^2}{dt} & \leq \bar{\psi}^2 + e^{4 + 4 \kappa} \frac{\Psi_{1 - t}(u(t))}{\tau} - (1 + \kappa)\bar{\psi}^2 \\
& \leq e^{4 + 4\kappa} \psi^2 - \kappa \bar{\psi}^2.
\end{split}
\end{equation*}
Note that $\frac{d}{ds} = 2 (1 - t) \frac{d}{dt},$ and 
the above reduces to
$$\frac{d \bar{\psi}^2}{ds} \leq 2\left( e^{4 + 4\kappa} \psi^2 - \kappa \bar{\psi}^2\right).$$
Since $\psi \leq \bar{\psi},$ we may use the product rule and divide out by $\bar{\psi},$ to obtain
$$\frac{d \bar{\psi}}{ds} \leq e^{4 + 4\kappa} \psi - \kappa \bar{\psi} .$$
By Gronwall's inequality, we have
\begin{equation*}
\begin{split}
\bar{\psi}(s) & \leq \left( \bar{\psi}(0)  + e^{4 + 4\kappa} \int_0^s \psi(\sigma) e^{ \kappa \sigma} \, d\sigma \right) e^{-\kappa s} \\
& \leq \eta e^{-\kappa s} + e^{4 + 4\kappa} \int_0^s \psi(\sigma) e^{\kappa (\sigma - s) } \, d \sigma. 
\end{split}
\end{equation*}
This establishes the claim (\ref{psiclaim}).

Now, let $\rho \in \LB \lambda, 1 \RB$ and take $\tilde{\tau}(t) = 3\rho^2 + 1 - t,$ so that $\tilde{\tau}(1 - \rho^2) = 4 \rho^2$ and $\tilde{\tau}'(t) = -1.$ 
Integrating (\ref{Psimonotonicity}) from $1 - \rho^2$ to $t \in \LB 1 - \rho^2, 1 \right),$ we obtain
$$\Psi_{\tilde{\tau}(t)}(u(t)) \leq \Psi_{\tilde{\tau}(1 - \rho^2)} \left( u \left( 1 - \rho^2 \right) \right) = \rho^2 \bar{\psi}^2(-\log \rho). $$ 
Since $3 \rho^2 \leq \tilde{\tau}(t) \leq 4 \rho^2,$ we get
\begin{equation}\label{dusquaredpreestimate}
\begin{split}
\int_{\frac12\rho }^{2 \rho } |du(t)|^2\, dV & \leq C \int_{\frac12 \rho }^{2 \rho } \frac{r^2}{\tilde{\tau}(t)} |du(t)|^2 e^{- \frac{r^2}{ 4 \tilde{\tau}(t) } } \, dV  \leq C \frac{\Psi_{\tilde{\tau}(t)}(u(t))}{\rho^2} \leq \bar{\psi}^2(- \log \rho).
\end{split}
\end{equation}
Assuming that $\rho^{2\kappa} \leq \frac{\eps_0}{\eta^2},$ (\ref{psiclaim}) implies that the RHS of (\ref{dusquaredpreestimate}) is less than $C \eps_0.$ We may apply standard epsilon-regularity, see e.g. \cite[Theorem 3.4]{songwaldron}, and (\ref{psiclaim}), to obtain the desired estimate.
\end{proof}

\begin{thm}\label{thm:mainthm}
Given $K \geq 1,$ $1 < \alpha \leq 2,$ and a compact Riemannian manifold $N,$ there exist constants $\eps_1 > 0$ and $C_{\ref{thm:mainthm}} \geq 1$ as follows.

Let $u : \R^2 \times \LB 0, T \right) \to N$ be a classical solution of harmonic map flow. Fix $0 < T_1 \leq T$ and $0 < R \leq \sqrt{T_1},$ and define $\eta$ by (\ref{etadef}).
Suppose that there exists $E_0 \geq 0,$ $0 < \eps \leq \eps_1,$ and $0 < \lambda \leq \frac12$ such that
\begin{equation}\label{mainthm:Phiassn}
\Phi_{R^2} \left( u(T_1 - R^2) \right) \leq E_0 + \eps, \qquad \Phi_{\lambda^2}\left( u(T_1 - (R\lambda)^2) \right) \geq E_0 - \eps,
\end{equation}
and,
for all $T_1 - R^2 \leq t < T_1 - (R\lambda)^2,$ $u(t)$ obeys a \L ojasiewicz inequality of the form
\begin{equation}\label{mainthm:lojasiewiczassumption}
| \Phi_{T_1 - t}(u(t)) - E_0| \leq K \left(\sqrt{T_1 - t} \| \That_{T_1 - t}(u(t)) \|_{T_1 - t} \right)^\alpha.
\end{equation}
Let $r_0,t,$ and $x,y\in \R^2$ be such that $R \sqrt{\lambda } \leq r_0 \leq \min \left\{ 1,  \left( C_{\ref{thm:mainthm}} \eta \right)^{-1} \right\} R,$ $T_1 - \frac14 r_0^2 \leq t < T_1,$ and $\frac{R^2 \lambda}{r_0} \leq |x|, |y| \leq r_0.$ If $\alpha < 2,$ we have
\begin{equation}\label{mainthm:alphabelowtwoestimate}
\left|  u(x,t) - u(y,t) \right| \leq C_{\ref{thm:mainthm}} \frac{| x - y |}{|x| + |y|} \left( \left( \eps^{\frac{\alpha - 2}{\alpha} } + \log \frac{R}{r_0} \right)^{\frac{\alpha - 1}{\alpha - 2} } + \eta \frac{ r_0}{R} \right). 
\end{equation}
If $\alpha = 2,$ we have
\begin{equation}\label{mainthm:alphaequalstwoestimate}
\left|u(x,t) - u(y,t) \right| \leq C_{\ref{thm:mainthm}} \frac{| x - y |}{|x| + |y|} \left( \sqrt{\eps} \left( \frac{r_0}{R} \right)^{\frac{1}{K}} + \eta \frac{ r_0}{R} \right). 
\end{equation}
\end{thm}
\begin{proof}
As before, we assume without loss of generality that $R = T_1 = 1.$ Let
$$\phi(t) := \Phi_{1 - t}(u(t)) - E_0.$$
By (\ref{Phimonotonicity}) and the \L ojasiewicz inequality (\ref{mainthm:lojasiewiczassumption}), we have
\begin{equation*}
\frac{d \phi(t)}{dt} + \frac{|\phi(t)|^{\frac{2}{\alpha} } }{K^{\frac{2}{\alpha}} (1 - t) } \leq 0.
\end{equation*}
As in the last proof, let
$$s = -\frac12 \log \left( 1 - t \right),$$
so that $t \in \LB 0, 1 - \lambda^2 \RB$ corresponds to $s \in \LB 0, S \RB,$ where
$$S = -\log \lambda.$$
We get
\begin{equation*}
\frac{d \phi(s)}{ds} + \frac{2 }{K^{\frac{2}{\alpha}}} |\phi(s)|^{\frac{2}{\alpha} } \leq 0.
\end{equation*}
Suppose $\alpha < 2$ and let
\begin{equation}\label{varphidef}
\varphi(s) : = \left( \eps^{-\frac{2 - \alpha}{\alpha} } + \frac{2(2 - \alpha)}{K^{\frac{2}{\alpha} } \alpha} s \right)^{-\frac{\alpha}{2 - \alpha} },
\end{equation}
so that $\varphi(0) = \eps$ and
$$\frac{d\varphi(s)}{ds} = - \frac{2}{K^{\frac{2}{\alpha} }} \varphi(s)^{\frac{2}{\alpha}}$$
for $s \geq 0.$ By Gronwall's inequality, we have
\begin{equation}\label{phiboundfromabove}
\phi(s) \leq \varphi (s).
\end{equation}
Meanwhile, we also have
$$\frac{d\left( - \varphi \left( S - s \right) \right)}{ds} = - \frac{2}{K^{\frac{2}{\alpha} }} \varphi \left( S - s \right)^{\frac{2}{\alpha}}$$
for $s \leq S.$
Applying Gronwall's inequality backwards from $s = S$ to $s = 0,$ we obtain
\begin{equation}\label{phivarphibound}
-\varphi \left( S - s \right) \leq \phi(s) \leq \varphi (s).
\end{equation}
Next, we translate this into an integral bound on $\Psi.$ Let
$$\delta(s) := \sqrt{1 - t} \| \That_{1 - t} \|_{1 - t},$$
where $1 - t = e^{-2s}.$
By (\ref{Phimonotonicity}), we have
$$\frac{d}{ds} \phi(s)^{\frac{\alpha - 1}{\alpha}} = - \frac{\alpha - 1}{\alpha} |\phi(s)|^{-\frac{1}{\alpha}} \delta^2(s),$$
where we define $\phi(s)^{\frac{\alpha - 1}{\alpha}} = -|\phi(s)|^{\frac{\alpha - 1}{\alpha}}$ if $\phi(s)$ is negative.
Fixing any $s \leq S/2,$ applying the \L ojasiewicz inequality (\ref{mainthm:lojasiewiczassumption}), we may write
\begin{equation}\label{integralofdeltaestimate}
\begin{split}
\int_{s}^{S - s} \delta(\sigma) \, d\sigma= \int_{s}^{S - s} \delta^2(\sigma) \delta(\sigma)^{-1} \, d\sigma & \leq \int_{s}^{S - s} \delta(\sigma)^2 |\phi(\sigma)|^{-\frac{1}{\alpha}} \, d\sigma \\
& \leq - \frac{\alpha}{\alpha - 1} \int_{s}^{S - s} \frac{d}{d\sigma} \phi(\sigma)^{\frac{\alpha - 1}{\alpha}} \, d\sigma \\
& \leq \frac{\alpha}{\alpha -1} \left( \phi(s)^{\frac{\alpha -1}{\alpha}} - \phi(S - s)^{\frac{\alpha -1}{\alpha}} \right) \\
& \leq \frac{2\alpha}{\alpha -1} \varphi(s)^{\frac{\alpha -1}{\alpha}},
\end{split}
\end{equation}
where we have used (\ref{phivarphibound}).

As in (\ref{psidefn}), define
$$\psi(s) = \sqrt{\frac{\Psi_{1 - t}(u(t))}{1 - t}},$$
where $1 - t = e^{-2s}$ as above.
By Lemma \ref{lemma:Poincare} with $\tau = 1 - t,$ we have
$$\psi(s) \leq 4 \delta(s).$$
Inserting this into  (\ref{integralofdeltaestimate}), we obtain
\begin{equation}\label{psiintegralbound}
\begin{split}
\int_{s}^{S - s} \psi(\sigma) \, d\sigma \leq 4 \int_{s}^{S - s} \delta(\sigma) \, d\sigma \leq \frac{8\alpha}{\alpha -1} \varphi(s)^{\frac{\alpha -1}{\alpha}}. 
\end{split}
\end{equation}
Note that by the monotonicity formula (\ref{Psimonotonicity}) with $\tau(t) = 1 - t,$ we have $\psi(s') \leq 4 \psi(s)$ for $s \leq s' \leq s + 1.$ 
Hence (\ref{psiintegralbound}) implies
$$\psi(s) \leq  C \eps^{\frac{\alpha -1}{\alpha}}$$
for $s \in \LB 1, S - 1\RB.$
In particular, assuming that $\eps$ is sufficiently small and advancing $s$ by one unit, the assumption (\ref{rduestimate:smallnessassumption}) is satisfied.
For $\lambda \leq r \leq (C\eta)^{-1}$ and $1 - \frac14 r^2 \leq t < 1,$ Lemma \ref{lemma:rduestimate} with $\kappa = 1$ now gives
\begin{equation}\label{rdupointwisebound}
r |du(x,t)| \leq  C \left( \int_0^{ s} \psi\left( \sigma \right) e^{ \sigma -s } \, d\sigma  + \eta r \right),
\end{equation}
where $s = - \log r.$

We now use (\ref{psiintegralbound}-\ref{rdupointwisebound}) to obtain the oscillation estimate (\ref{mainthm:alphabelowtwoestimate}). Let $r_0 \in \LB \sqrt{\lambda}, 1 \RB$ be given and put
$$s_0 = - \log r_0 \in \LB 0 , S/2\RB.$$
Let $1 - \frac14 r_0^2 \leq t < 1$ and $x, y$ with $\frac{\lambda}{r_0} \leq |x|, |y| \leq r_0.$

\vspace{2mm}

\noindent {\bf Case 1.} The points $x$ and $y$ are contained in a dyadic annulus $U_{r_1}^{2r_1},$ where $\frac{\lambda}{r_0} \leq r_1 \leq \frac{r_0}{2}.$

Let $s_1 = - \log r_1 \in \LB s_0, S - s_0 \RB.$ We can integrate (\ref{rdupointwisebound}) along a path, to obtain 
\begin{equation}\label{rdulaterpointwisebound}
r_1 |u(x, t) - u(y, t) | \leq C |x - y| \left( \int_0^{s_1} \psi\left( s \right) e^{ s - s_1 } \, ds  + \eta r_1 \right).
\end{equation}
Since $s_0 \leq s_1,$ we have
\begin{equation}\label{intpsifirstest}
\begin{split}
\int_0^{s_1} \psi \left( s \right) e^{ s - s_1 } \, ds & \leq \int_0^{s_0} \psi(s) e^{s - s_0} \, ds + \int_{s_0}^{S - s_0} \psi(s) \, ds \\
& \leq C \left( \varphi(0)^{\frac{\alpha - 1}{\alpha} } e^{-s_0} + \int_0^{s_0} \varphi(s)^{\frac{\alpha - 1}{\alpha} } e^{s - s_0} \, ds + \varphi \left( s_0 \right)^{\frac{\alpha - 1}{\alpha}} \right),
\end{split}
\end{equation}
where we have used integration by parts and (\ref{psiintegralbound}). From the definition (\ref{varphidef}), for $s \leq s_0,$ we may calculate
\begin{equation}\label{varphiss0est}
\begin{split}
\varphi(s) = \varphi(s_0) \frac{\varphi(s)}{\varphi(s_0)} & = \varphi(s_0) \left( \frac{\eps^{- \frac{2 - \alpha}{\alpha} } + Cs_0  }{\eps^{- \frac{2 - \alpha}{\alpha} } + Cs} \right)^{\frac{\alpha}{2 - \alpha} } \\
& = \varphi(s_0) \left( 1 + \frac{C (s_0 - s)}{\eps^{- \frac{2 - \alpha}{\alpha}} + Cs } \right)^{\frac{\alpha}{2 - \alpha} } \\
& \leq C \varphi(s_0) \left( 1 + C \eps^{\frac{2 - \alpha}{\alpha} } (s_0 - s) \right)^{\frac{\alpha}{2 - \alpha} }.
\end{split}
\end{equation}
Inserting into (\ref{intpsifirstest}), and using $\varphi(0)^{\frac{\alpha - 1}{\alpha}} e^{-s_0} \leq \varphi(s_0)^{\frac{\alpha - 1}{\alpha}},$ we get
\begin{equation*}
\begin{split}
\int_0^{s_1} \psi \left( s \right) e^{ s - s_1 } \, ds 
& \leq C \varphi(s_0)^{\frac{\alpha - 1}{\alpha}} \left( 1 + \int_0^{s_0} \left( 1 + C (s_0 - s) \right)^{\frac{\alpha - 1}{2 - \alpha} } e^{s - s_0} \, ds \right) \\
& \leq C \varphi(s_0)^{\frac{\alpha - 1}{\alpha}} .
\end{split}
\end{equation*}
Inserting into (\ref{rdulaterpointwisebound}) gives (\ref{mainthm:alphabelowtwoestimate}).

\vspace{2mm}

\noindent {\bf Case 2.} The points $x$ and $y$ are not contained in a dyadic annulus.

By Case 1, it suffices to assume that $x$ and $y$ lie along a ray, and $2|x| \leq |y|.$ We have
$$|u(x, t) - u(y, t) | \leq \int^{r_0}_{\lambda / r_0} |du| \, dr = \int_{s_0}^{S - s_0} r |du| \, ds.$$
Inserting (\ref{rdupointwisebound}), we have 
\begin{equation}\label{oscUestimate}
\begin{split}
|u(x, t) - u(y, t) | \leq \int_{s_0}^{S - s_0} r |du| \, ds & \leq C \int_{s_0}^{S - s_0} \left( \int_0^{s} \psi \left( \sigma \right) e^{ \sigma - s } \, d\sigma  + \eta e^{-s} \right) \, ds \\
& \leq C \left( \int_{s_0}^{S - s_0} \!\!\! \int_0^{s} \psi\left( \sigma \right) e^{ \sigma - s } \, d\sigma \, ds  + \eta e^{-s_0} \right).
\end{split}
\end{equation}
Changing variables $\varsigma = s - \sigma,$ we can estimate the integral by
\begin{equation*}
\begin{split}
\int_{s_0}^{S - s_0} \!\!\! \int_0^{s} \psi\left( \sigma \right) e^{\sigma - s} \, d\sigma \, ds & = \int_{0}^{S - s_0} e^{-\varsigma} \int_{\max \{s_0 - \varsigma, 0\} }^{S - s_0 - \varsigma} \psi(\sigma ) \, d\sigma d\varsigma \\
& \leq C \int_{0}^{S - s_0} e^{- \varsigma} \varphi \left( \max \{s_0 - \varsigma, 0 \} \right)^{\frac{\alpha - 1}{\alpha}} d\varsigma,
\end{split}
\end{equation*}
where we have used (\ref{psiintegralbound}). Using (\ref{varphiss0est}) again, with $s = s_0 - \varsigma,$ we can estimate
\begin{equation*}
\begin{split}
 \int_{0}^{S - s_0} e^{- \varsigma} \varphi \left( \max \{s_0 - \varsigma, 0 \} \right)^{\frac{\alpha - 1}{\alpha}} d\varsigma & = \int_{0}^{s_0} e^{- \varsigma} \varphi \left( s_0 - \varsigma \right)^{\frac{\alpha - 1}{\alpha}} d\varsigma  + \int_{s_0}^{S - s_0} e^{- \varsigma} \varphi \left( 0 \right)^{\frac{\alpha - 1}{\alpha}} d\varsigma  \\
 & \leq C \varphi(s_0 )^{\frac{\alpha - 1}{\alpha}} \left( \int_0^{s_0} e^{-\varsigma} \left( 1 + C \varsigma \right)^{\frac{\alpha - 1}{2 - \alpha}} \, d \varsigma + 1 \right) \\
 & \leq C \varphi(s_0 )^{\frac{\alpha - 1}{\alpha}} .
\end{split}
\end{equation*}
Returning to (\ref{oscUestimate}), we obtain 
\begin{equation*}
\begin{split}
|u(x, t) - u(y, t) |  \leq C \left(  \varphi \left( s_0 \right)^{\frac{\alpha - 1}{\alpha}} + \eta r_0 \right),
\end{split}
\end{equation*}
which is the bound (\ref{mainthm:alphabelowtwoestimate}) in the present case. 
This completes the proof of (\ref{mainthm:alphabelowtwoestimate}).

The bound (\ref{mainthm:alphaequalstwoestimate}) is proved similarly.
\end{proof}

\begin{cor}\label{cor:oscillation}
Suppose $T < \infty$ and take $T_1 = T$ in the previous theorem. If the assumptions (\ref{mainthm:Phiassn}-\ref{mainthm:lojasiewiczassumption}) hold for all $\lambda \in \left( 0, R \RB,$ then the limit $u(x,T) := \lim_{t \nearrow T} u(x,t)$ exists for all $x$ sufficiently close to $0,$ 
with $u(T) \in C^{1 / K}$ if $\alpha = 2$ and
\begin{equation*}
\Osc_{B_r(0) \setminus \{0\}} u(T) = O \left| \log r \right|^{\frac{\alpha - 1}{\alpha - 2}}
\end{equation*}
as $r \searrow 0$ if $1 < \alpha < 2.$ In particular, $u(T)$ extends over the origin to a continuous map.
\end{cor}
\begin{proof} This follows by letting $\lambda \searrow 0$ in the previous theorem.
\end{proof}

\begin{rmk}
H{\"o}lder continuity in the case $\alpha = 2$ agrees with Topping's result in \cite{toppingholder}.
\end{rmk}

\begin{cor} 
Let $u: \R^2 \times \LB 0, T \right) \to S^2$ be a classical solution of harmonic map flow. 

\vspace{2mm}

\noindent (a) The body map $u(T)$ extends to a continuous map $\R^2 \to S^2.$

\vspace{2mm}

\noindent (b) Given a point $p \in \R^2$ and any sequence of times $t_i \nearrow T,$ the maps $u(t_i)$ sub-converge in the bubble-tree sense without necks. Specifically, there exist finitely many nonconstant harmonic maps $\phi_k : \R^2 \to S^2,$ obtained as rescaled limits of $u(t_i),$ such that the energy identity holds:
$$\lim_{r \searrow 0} \lim_{t \nearrow T} E(u(t), B_r(p) ) = \sum_k E(\phi_k).$$
Moreover, we have $\lim_{x \to \infty} \phi_k(x) = u(p,T)$ for at least one $k.$

\end{cor}
\begin{proof} ($a$) It suffices to prove continuity of $u(T)$ at $0 \in \R^2.$ By monotonicity of $\Phi,$ the limit
$$E_0 : = \lim_{R \searrow 0} \Phi_{R^2} \left( u(T - R^2) \right) \geq 0$$
exists, and we may take $R > 0$ such that
$$\Phi_{R^2}(u(T - R^2)) \leq E_0 + \eps_0.$$
Now, the \L ojasiewicz inequality of Theorem \ref{thm:firstloj} gives (\ref{mainthm:lojasiewiczassumption}) after rescaling by $\tau = T - t,$ where we must have $E_0 = 4 \pi n.$ 
The result follows from the previous Corollary.

\vspace{2mm}

\noindent ($b$) The proof is standard in light of Theorem \ref{thm:mainthm} (see \cite[Theorem 6.8]{songwaldron} for a recent account), and we omit it.
\end{proof}

\bibliographystyle{amsinitial}
\bibliography{biblio}

\end{document}